\documentclass[a4paper,12pt]{article}
\usepackage[utf8]{inputenc}
\usepackage{amssymb}
\usepackage{ifthen}

\makeatletter
\def\@seccntformat#1{\csname the#1\endcsname.\ } % the column after a section number
\makeatother

\newif\ifNoRemark
\def\addtheorem#1#2#3#4{
\ifthenelse{\equal{#2}{}}{}%
{\ifthenelse{\expandafter\isundefined\csname the#2\endcsname}{\newcounter{#2}}{}}
\newenvironment{#1}[1][\global\NoRemarktrue]% No Remark by default
{\par\addvspace{2mm plus 0.5mm minus 0.2mm}\noindent % new paragraph without indent
{\bf #3}\ifthenelse{\equal{#2}{}}{}%
{\refstepcounter{#2}{\bf ~\csname the#2\endcsname}}%
{\bf \vphantom{##1}\ifNoRemark.\ \else\ (##1).\fi}\begingroup #4}%
{\endgroup\par\addvspace{1mm plus 0.5mm minus 0.2mm}\global\NoRemarkfalse}
\expandafter\newcommand\csname b#1\endcsname{\begin{#1}}
\expandafter\newcommand\csname e#1\endcsname{\end{#1}}
}

\addtheorem{theorem}{thrm}{Theorem}{\sl}
\addtheorem{lemma}{lmm}{Lemma}{\sl}
\addtheorem{proposition}{prpstn}{Proposition}{\sl}
\addtheorem{corollary}{crll}{Corollary}{\sl}
\addtheorem{problem}{prblm}{Problem}{\sl}
\addtheorem{remark}{rmrk}{Remark}{}
\addtheorem{example}{exmpl}{Example}{}
\newenvironment{proof}[1][\hspace{-1.0ex}]%
{\par\addvspace{1mm}{\sc Proof\hspace{1.0ex}{#1}.} }%
{\quad$\blacktriangle$\par\addvspace{1mm}}
\date{}

\author{D.~S.~Krotov, I.~{Yu}.~Mogilnykh, V.~N.~Potapov}
\title{To the theory of $q$-ary Steiner and other-type trades%
\thanks{This research was financed
 by the Russian Science Foundation (grant No 14-11-00555).
The authors are with the Sobolev Institute of Mathematics,
pr. Akademika Koptyuga 4, Novosibirsk 630090, Russia,
and with the Novosibirsk State University, Pirogova 2, Novosibirsk 630090, Russia.
}%
}

\begin{document}
\def\diam{\mathrm{D}}
\maketitle
\begin{abstract}
 We introduce the concept of a clique bitrade, which generalizes several known types of bitrades, including latin bitrades, Steiner $\mathrm{T}(k-1,k,v)$ bitrades, extended $1$-perfect bitrades. For a distance-regular graph, we show a one-to-one correspondence between the clique bitrades that meet the weight-distribution lower bound on the cardinality and the bipartite isometric subgraphs that are distance-regular with certain parameters. As an application of the results, we find the minimum cardinality of $q$-ary Steiner $\mathrm{T}_q(k-1,k,v)$ bitrades and show a connection of minimum such bitrades with dual polar subgraphs of the Grassmann graph $J_q(v,k)$.

 \end{abstract}

\section{Introduction}
In this paper,
we prove some results on
a rather general class of combinatorial bitrades,
mainly concentrating on bitrades of minimum possible cardinality,
and obtain a partial result concerning minimum $q$-ary Steiner bitrades.

As a quick-start introduction to the terminology used in this paper,
we consider it by the example of
the Steiner triple systems, well-known in combinatorics.
Let $T$ be the set of $3$-subsets, called triples,
of a given finite set $V$ of cardinality  $v\ge 6$.
We consider a graph on the vertex set $T$,
where two triples are adjacent if and only if they intersect in two elements
(this graph is known as the Johnson graph $J(n,3)$).
In this graph, the triples that include a given $2$-subset form a maximal
clique (a set of mutually adjacent vertices).
A subset of $T$ that intersects with every such clique
in exactly one vertex is known as a Steiner triple system, or STS$(v)$, or S$(2,3,v)$.
A \emph{Steiner bitrade} (of type T$(2,3,v)$) is defined as a pair $(T_0,T_1)$ of disjoint nonempty subsets (called \emph{trades}) of $T$
such that every $2$-subset of $V$ is either included in exactly one triple from $T_0$
and exactly one triple from $T_1$ or is not included in any triple of $T_0\cup T_1$.
Equivalently, every maximum clique intersects with each of $T_0$, $T_1$ in exactly one element
or does not intersect with both of them.
Given two different STS$(v)$ $S$ and $S'$, the difference pair $( S\backslash S', S'\backslash S)$
is necessarily a Steiner bitrade; however, not every Steiner bitrade can be obtained in such a way.
%moreover, Steiner bitrades exist for every $v\ge 6$, while necessary condition for existence of an STS$(v)$
%is $v \equiv 1,3 \bmod 6$ \cite{???}.
%This gives additional possibilities for the study of the variety of Steiner bitrades.
The graph terminology used in this paragraph to treat the Steiner triple systems
and the Steiner bitrades allows to define similar structures for other graphs,
covering some other well-known classes of objects
(perfect codes, latin squares and hypercubes, and the corresponding bitrades,
see the examples in Section~\ref{sect:known}),
as well as classes that are studied
 intensively only over the last years
($q$-ary Steiner systems).

Bitrades (trades) are used in combinatorics,
including combinatorial design theory and combinatorial coding theory,
to study possible differences between two combinatorial
objects from the same class and to obtain new objects from a given one.
In particular, small bitrades are used to construct large classes of objects
with the same parameters, see e.g.
\cite{Vas:nongroup_perfect.en}, \cite{AssMat:66},
\cite{Aliev:72}, \cite{ZinZin:2009}, \cite{PotKro:numberQua.en},
\cite{HK:q-ary};
often minimum trades are utilized
to get a lower bound on the number of objects.
Trades embedded into ``complete'' objects
are also known under term
``switching components'' \cite{Ost:2012:switching};
but in general, bitrades (and trades)
are defined independently.
%%% suggest removing the final part of the sentence since context would be unclear for nonspecialists
% using the local relation
% that defines the ``complete'' objects.
Trades can exist even if complete objects
with the corresponding parameters do not exist (for example, $v\not\equiv 1,3 \bmod 6$ for STS$(v)$).
This gives an additional motivation
to study trades and, as a result,
to develop the theory
of ``complete'' objects of given type.

One interesting and important class
of such ``complete'' objects
is the class of subspace designs,
which are also known
as $q$-ary generalizations
of classic combinatorial designs.
In particular,
if we replace in the definition of STS$(v)$ above
the set $V$ by a $v$-dimensional space over GF$(q)$
and the $t$-subsets by $t$-dimensional subspaces,
then we get the definition of $\mathrm{STS}_q(v)$,
a $q$-ary Steiner system.
It should be mentioned that the increasing interest
to subspace codes and subspace designs
in the recent years is partially motivated
by their importance for network coding applications, see e.g. \cite{KK:2008}.
Among the series of new results, we mention
\cite{BEOVW:q-Steiner}
(nontrivial $\mathrm{STS}_q(v)$ do exist)
and \cite{FLV:2014}
(nontrivial $q$-ary simple $t$-designs exist for all $t$).
While the trades, in the case of subspace designs,
were not considered independently before,
an equivalent concept (so-called $t$-equivalent sets)
was used for the construction of subspace designs \cite{BKKL:large}.
The current paper, apart from some general results,
which establish common properties of several partial kinds of trades,
contains a contribution to the theory of subspace designs.
Namely, we find the minimum possible cardinality of a $q$-ary Steiner trade
of type $\mathrm{T}_q(d-1,d,n)$.

In Section~\ref{sect:general},
we define the main notations and concepts (Subsection~\ref{sect:def}),
including the concept of a clique bitrade,
and prove four general theorems.
 Theorem~\ref{th:bitrade} (Subsection~\ref{sect:trade-criterion}) shows
that the existence of a clique bitrade
in a regular graph
is equivalent to the existence
of an eigenfunction
with certain restrictions
and to the existence
of a bipartite regular subgraph
of certain degree.
 In Subsection~\ref{sect:Delsarte},
we recall the concepts of Delsarte cliques and Delsarte pairs
and establish some useful properties of the Delsarte cliques.
As a corollary, we prove an intersecting characterization of the eigenfunctions
with the minimum eigenvalue, related to these concepts (Theorem~\ref{th:eigen}).
 In Subsection~\ref{sect:wd},
 we consider the weight-distribution lower bound on the number
of nonzeros of an eigenfunction.
 Theorem~\ref{th:min} in Subsection~\ref{sect:min} shows
that in the case of a distance-regular graph the existence of a clique bitrade meeting
the weight-distribution lower bound
is equivalent to the existence
of a bipartite regular isometric subgraph
of certain degree.
 Theorem~\ref{th:d-r} states that
the isometric subgraph mentioned above
is distance-regular.

In Section~\ref{sect:known}
we illustrate the theory by examples of bitrades
already known in the literature,
including ones
from design theory (Example~\ref{x:Johnson}),
coding theory (Example~\ref{x:halved}),
the theory of latin squares and latin hypercubes (Example~\ref{x:Hamming}).

In Section~\ref{sect:q-ary},
based on the results of Section~\ref{sect:general}
and known facts about the dual polar graphs,
we find $q$-ary Steiner $\mathrm{T}_q(d-1,d,n)$ bitrades
of minimum cardinality.

%===========================================
%===========================================
%===========================================
%===========================================
%===========================================
\section{General theory}\label{sect:general}
\subsection{Basic definitions}\label{sect:def}
Given a connected graph $\Gamma$,
by the \emph{distance} $\mathrm{d}_G(x,y)$
between two vertices $x$ and $y$, we mean
the length of a shortest path from $x$ to $y$.
For a graph $\Gamma=(V,E)$ and a vertex $x\in V$
or a set of vertices $x\subset V$,
$\Gamma_i(x)$
denotes
the $i$th neighborhood of $x$, that is,
the set of vertices at distance $i$ from $x$.
The \emph{diameter} $\diam(\Gamma)$ of $\Gamma$
is the maximum distance
between two vertices of $\Gamma$.

An \emph{eigenfunction} of a graph $\Gamma=(V,E)$
is a function $f:V\to \mathbb{R}$
that is not constantly zero and satisfies
\begin{equation}\label{eq:eigendef}
  \sum_{y\in \Gamma_1(x)} f(y)=\theta f(x)
\end{equation}
for all $x$ from $V$ and some constant $\theta$,
which is called an \emph{eigenvalue} of $\Gamma$.
The eigenfunctions of a graph can be treated
as the eigenvectors
of its adjacency matrix.

A set $C$ of vertices of a regular graph $\Gamma$ of degree $k$ is said to be
\emph{completely regular}
with \emph{covering radius} $\rho$
if $\Gamma_\rho(C) \ne \emptyset = \Gamma_{\rho+1}(C)$ and there is a sequence
$(b_0,\ldots,\linebreak[1]b_{\rho-1};\linebreak[2]
c_1,\ldots,\linebreak[1]c_{\rho})$,
which is named the \emph{intersection array},
such that
$\Gamma_{i+1}(C)\cap \Gamma_1(y) = b_i$
and $\Gamma_{i-1}(C)\cap \Gamma_1(y) = c_i$
hold for every
$i\in\{0,\ldots,\rho \}$,
and every $y\in \Gamma_i(C)$,
where $b_{\rho}=c_0 = 0$.
The numbers $b_0$, \ldots, $b_\rho$, $c_0$, \ldots, $c_\rho$, and
$a_0$, \ldots, $a_\rho$, where $a_i=k-b_i-c_i$, are referred to as
the \emph{intersection numbers},
and the tridiagonal matrix $(a_{i,j})_{i,j=0}^\rho$,
where $a_{i,i}=a_i$, $a_{i,i+1}=b_i$, $a_{i,i-1}=c_i$,
is called the \emph{intersection matrix} of $C$.
By the eigenvalues of a completely regular set,
we will
 mean the eigenvalues of its intersection matrix.
Given a completely regular set $C$
of covering radius $\rho$
and one of  its eigenvalues $\theta$,
by $\delta^{\theta}_C$
we denote the function on the vertex set that equals
$\nu_i$ on $\Gamma_i(C)$,
where $(1=\nu_0,\nu_1,\ldots,\nu_\rho)$
is an eigenvector
of the intersection matrix
corresponding to the eigenvalue $\theta$
(it is easy to see for a tridiagonal matrix with nonzero lower- and upper-diagonal elements
that any eigenvector is uniquely determined by its first element and the eigenvalue,
which means that for each eigenvalue there is a unique eigenvector starting with $1$;
in particular, there are $\rho+1$ different eigenvalues).
It is straightforward that $\delta^{\theta}_C$ is an eigenfunction of the graph
with the same eigenvalue $\theta$,
which proves the known fact that an eigenvalue
of a completely regular set
is necessarily an eigenvalue of the graph.

A connected graph $\Gamma$
is called \emph{distance-regular}
if every singleton is completely regular with the same intersection array
(independent on the choice of the vertex),
which is called the intersection array of $\Gamma$.

Let $\Gamma$ be a connected regular graph of degree $k$.
Assume that $S$ is a set of $(s+1)$-cliques
in $\Gamma$
such that
%every clique from $S$ 
%has order $s+1$, 
every edge of $\Gamma$ is included
in exactly $m$ cliques from $S$;
in this case, we
will say that the pair $(\Gamma,S)$ is a \emph{$(k,s,m)$ pair}.
A couple $(T_0,T_1)$ of mutually disjoint nonempty sets of vertices
is called an \emph{$S$-bitrade},
or a \emph{clique bitrade},
if every clique from $S$ either intersects
with each of $T_0$ and $T_1$ in exactly one vertex
or does not intersect
with both of them
(in particular, this means
that each of $T_0$, $T_1$
is an independent set in $\Gamma$).
A set of vertices $T_0$ is called an \emph{$S$-trade}
if there is another set $T_1$ (known as a \emph{mate} of $T_0$)
such that the pair $(T_0,T_1)$ is an $S$-bitrade.

%-------------------------------------
%-------------------------------------
%-------------------------------------
\subsection{A bitrade criterion}\label{sect:trade-criterion}
We start with a criterion, which can be used as alternative definition of a clique bitrade.
\begin{theorem}\label{th:bitrade}
  Let $\Gamma$ be a regular graph of degree $k$.
  Let $(\Gamma,S)$, where $S$ is a set of cliques in $\Gamma$, be a $(k,s,m)$ pair.
  Let $T=(T_0,T_1)$ be a pair of disjoint nonempty independent sets of vertices of $\Gamma$.
  The following assertions are equivalent.

  {\rm (a)} $T$ is an $S$-bitrade.

  {\rm (b)} The function
  \begin{equation}\label{eq:eigen}
    f^T(x)=\left\{
    \begin{array}{ll}
      (-1)^i & \mbox{if $\bar x\in T_i$, $i\in\{0,1\}$} \\
      0 & \mbox{otherwise}
    \end{array}\right.
  \end{equation}
  is an eigenfunction of $\Gamma$ with eigenvalue $\theta = -k/s$.

  {\rm (c)} The subgraph $\Gamma^T$ of $\Gamma$ generated by the vertex set
      $T_0 \cup T_1$ is regular with degree $-\theta = k/s$
      (as $T_0$ and $T_1$ are independent sets, this subgraph is bipartite).
\end{theorem}
\begin{proof}
All proofs are based on counting arguments,
mainly utilizing the definition
of a $(k,s,m)$ pair.

(a)$\Rightarrow$(b):
Assume $(T_0,T_1)$ is an $S$-bitrade.

At first, we show that (\ref{eq:eigendef}) holds
for every vertex $x\not\in T_0\cup T_1$.
Indeed, double-counting the number
of triples $(u,t_0,t_1)$ such that
$t_0$, $t_1$, $x\in u\in S$ and
$t_l \in \Gamma_1(x) \cap T_l$ gives
$$|\Gamma_1(x) \cap T_0|\cdot m = |\Gamma_1(x) \cap T_1|\cdot m.$$
Thus,  (\ref{eq:eigendef}) holds
with both sides being equal to zero.

It remains to prove (\ref{eq:eigendef})
for $x \in T_0$
(the case $x \in T_1$ is similar).
Double-counting the number
of triples $(u,t,t_1)$ such that
$t,t_1,x\in u\in S$,
$t \in \Gamma_1(x) \backslash T_1$ and
$t_1 \in \Gamma_1(x) \cap T_1$ gives
$$|\Gamma_1(x) \cap T_1|\cdot m \cdot (s-1)
=|\Gamma_1(x) \backslash T_1|\cdot m \cdot 1$$
(for the left side, we choose $t_1$ first
then $u$ containing $x$ and $t_1$, and finally 
 $t$ from $u\backslash \{x,t_1\}$; the right side corresponds to the order $t$, $u$, $t_1$).
This implies $|\Gamma_1(x) \cap T_1|=|\Gamma_1(x)|/s = k/s$.
Since $T_0$ is independent and therefore we have
 $|\Gamma_1(x) \cap T_0|=0$,
we find that (\ref{eq:eigendef}) turns to $\theta=\theta$.

(b)$\Rightarrow$(c) is trivial.

(c)$\Rightarrow$(a): Let $\Gamma^T$ be regular of degree $k/s$.
Let us consider some $x$ from $T_0$.
There are $km/s$ cliques from $S$ containing $x$
(as well as any other fixed vertex).
On the other hand, every
$y\in \Gamma_1(x)\cap T_1$
is in $m$ of them.
Since, by the hypothesis,
there are $k/s$ such $y$,
every clique from $S$ containing $x$
contains some $y\in \Gamma_1(x)\cap T_1$.
By the definition of an $S$-bitrade,
the claim follows.
\end{proof}
%=====================================================
\subsection{Delsarte cliques, Delsarte pairs, and eigenfunctions}\label{sect:Delsarte}
We can say more if $\Gamma$ is distance-regular.
It is known \cite{Hoffman:1970} (see also \cite[Proposition 4.4.6]{Brouwer})
that a clique in a distance-regular graph cannot have more than $1-k/\theta_{\min}$ elements,
where $\theta_{\min}$ is the minimum eigenvalue of the graph;
the cliques of cardinality $1-k/\theta_{\min}$ are called \emph{Delsarte cliques}.
A $(k,s,m)$  pair $(\Gamma,S)$  is known as a \emph{Delsarte pair}
if the graph $\Gamma$ is distance-regular
and $S$ consists of Delsarte cliques \cite{BHK:2007}.
\begin{proposition}\label{p:mx-mn}
If, under notation and hypothesis of Theorem~\ref{th:bitrade},
{\rm(a)}--{\rm(c)} hold and, additionally,
the graph $\Gamma$ is distance-regular,
then $\theta$ is the minimum eigenvalue of $\Gamma$,
$s+1$ is the maximum order of a clique in $\Gamma$,
and $(\Gamma,S)$ is a Delsarte pair.
\end{proposition}
\begin{proof}
  As proved in \cite{Hoffman:1970},
any clique of order $M$ satisfies $M \le 1-k/\theta_{\min}$,
where $\theta_{\min}$ is the minimum eigenvalue of $\Gamma$.
From (b) we have $s+1 = 1-k/\theta$.
Thus, $\theta=\theta_{\min}$ and $s+1$ is the maximum order of a clique in $\Gamma$.
The pair $(\Gamma,S)$ is a Delsarte pair by the definition.
\end{proof}
We will use the following properties of the Delsarte cliques:
\begin{proposition}\label{p:Delsarte-clique}
Let $\Gamma$ be a distance-regular graph of degree $k$ and diameter $D$,
and let $\theta$ be the minimum eigenvalue of $\Gamma$.
Then

(i) every Delsarte clique is a completely-regular set of covering radius $D-1$,
and $\theta$ is not one of its eigenvalues;

(ii) the sum of any eigenfunction with eigenvalue $\theta$ over a Delsarte clique is zero;

(iii) there are positive numbers
$s_0^+$, \ldots, $s_{D-1}^+$,
$s_1^-$, \ldots, $s_{D}^-$ such that
for every vertex $x$ and every Delsarte clique $C$ at distance $i$ from $x$,
there hold $|\Gamma_i(x)\cap C |=s_i^+$ and $|\Gamma_{i+1}(x)\cap C |=s_{i+1}^-$.
\end{proposition}
\begin{proof}
It is shown in \cite[Proposition 4.4.6, Remark]{Brouwer}
that a Delsarte clique is a completely-regular set
and its covering radius $\rho$ is less than $D$;
hence, $\rho=D-1$.
The second claim of (i)
can also be retrieved
from the proof of \cite[Proposition 4.4.6]{Brouwer}
and some algebraic background,
but we will give another proof.
Let $\theta_0$, \ldots, $\theta_{D-1}$, $\theta'$
be the eigenvalues of $\Gamma$,
where $\theta'$ is not an eigenvalue of the given Delsarte clique $C$.
Then the characteristic function $\chi_C$ of $C$ is a linear combination of
$\delta^{\theta_0}_C$, \ldots, $\delta^{\theta_{D-1}}_C$
(as well as the vector $(1,0,...,0)$ is a linear combination of the eigenfunctions
of the intersection matrix of $C$).
Thus, $\chi_C$ is orthogonal to any eigenfunction with eigenvalue $\theta'$;
in other words, the sum of the values of any eigenfunction with eigenvalue $\theta'$ over $C$
is zero. Let us consider the eigenfunction $\delta^{\theta'}_{\{x\}}$ for some $x$
from $C$. We have $\delta^{\theta'}_{\{x\}}(x)=1$ and
$\delta^{\theta'}_{\{x\}}(y)=\theta'/k$
for all $y$ from $C\backslash\{x\}$.
Since, by the definition
of a Delsarte clique,
$|C\backslash\{x\}|=-k/\theta$,
we have $1+(k/\theta')(-k/\theta)=0$.
So, $\theta'=\theta$, and (i) and (ii) hold.

(iii) Let $\nu^\theta=(\nu_0,\nu_1,\ldots,\nu_D)$
be the eigenvector
of the intersection matrix of $\Gamma$
corresponding to $\theta$
and starting with $1$.
Let us consider a vertex $x$
at distance $i$ from
the given Delsarte clique $C$.
Since the sum of the values of %the eigenfunction 
$\delta^\theta_{\{x\}}$
over $C$ is zero,
we have
\begin{equation}\label{eq:GiGi1}
|\Gamma_i(x)\cap C |\nu_i+|\Gamma_{i+1}(x)\cap C |\nu_{i+1}=0.
\end{equation}
If the first summand of the left part is nonzero,
then the second one is nonzero too;
it follows by induction on $i$ that all $\nu_i$, $i=0,1,...,D$, are nonzero.
Then, (\ref{eq:GiGi1}) implies that
$|\Gamma_i(x)\cap C |/|\Gamma_{i+1}(x)\cap C |=-\nu_{i+1}/\nu_{i}$. Claim (iii) follows.
\end{proof}

As a corollary,
for the distance-regular graphs,
we can formulate
a stronger analog of
the equivalence of (a) and (b)
in Theorem~\ref{th:bitrade}:

\begin{theorem}\label{th:eigen}
 Let $(\Gamma,S)$ be a Delsarte pair and $\theta$
 be the minimum eigenvalue of $\Gamma$.
 Then

 (i) A function $f$ over the vertex set of $\Gamma$ is an eigenfunction
 with the eigenvalue $\theta$
 if and only if
 for every clique $C$ from $S$
 it holds $\sum_{x\in C} f(x)=0$.

 (ii) A proper subset $B$ of the vertex set of $\Gamma$ is
 a completely regular set of radius $1$ with eigenvalue $\theta$
 if and only if
 it has a constant number of elements in any clique from $B$.
\end{theorem}
\begin{proof}
 In (i), the ``if'' statement follows from direct checking of (\ref{eq:eigendef}),
 while ``only if'' comes from Proposition~\ref{p:Delsarte-clique}(ii).
 In (ii), again, ``if'' is straightforward, while ``only if'' follows from claim (i)
 if we consider the eigenfunction $\delta^{\theta}_B$.
\end{proof}
%--------------------------------------------------------
%--------------------------------------------------------
%--------------------------------------------------------
\subsection{The weight-distribution bound}\label{sect:wd}
The rest of Section~\ref{sect:general}
is devoted to the bitrades that are minimum in the sense that
their cardinality meets a special lower bound.
In this subsection, we define the weight distribution
and the weight-distribution bound.

By the \emph{weight distribution}
of a function
$f:V\to \mathbb{R}$
with respect to a vertex $x$
of a graph $\Gamma=(V,E)$
we will mean the sequence
$W(x)=(W^i(f))_{i=0}^{\diam(\Gamma)}$,
where $W^i(f)=\sum_{y\in \Gamma_i(x)} f(y)$.
The following fact
is well known
and easy to derive
from definitions,
by induction on $i$.
\begin{lemma}\label{l:w}
  The weight distribution $W(x)$
  of an eigenfunction $f$
  of a distance-regular graph $\Gamma$
  is calculated as
  $(f(x)W^i_{A,\theta})_{i=0}^{\diam(\Gamma)}$
  where the coefficients
  $W^i_{A,\theta}$
  are derived from
  the intersection array $A=(b_0,\ldots,c_{\diam(\Gamma)})$ of $\Gamma$ and
  the eigenvalue $\theta$
  that corresponds to $f$:
  \begin{equation}\label{eq:rekur}
   W^0_{A,\theta}=1, \quad
   W^1_{A,\theta}=\theta, \quad
   W^{i}_{A,\theta}=\left((\theta-a_{i-1})W^{i-1}_{A,\theta}-b_{i-2} W^{i-2}_{A,\theta}\right)/c_i, \quad i\ge 2.
  \end{equation}
\end{lemma}
To read more about
how to calculate
the weight distribution of eigenfunctions
and generalizations of eigenfunctions,
see \cite{Kro:struct}.
Using Lemma~\ref{l:w},
it is easy to derive
the following lower bound
on the support of an eigenfunction.
The bound is also known;
a partial case of this argument
was used in \cite{EV:94}
to find the minimum cardinality
of a switching component
of binary $1$-perfect codes
(which can also be treated
in terms of eigenfunctions).
\begin{corollary}[the weight-distribution (w.d.) bound]\label{cor:bound}
  An eigenfunction $f$
  of a dis\-tan\-ce-re\-gu\-lar graph
  has at least
  $\sum_{i=0}^{\diam(\Gamma)} |W^i_{A,\theta}|$ nonzeros,
  in notation of Lemma~\ref{l:w}.
\end{corollary}
\begin{proof}
  Considering the weight distribution
  with respect to a vertex $x$
  with the maximum value of $|f(x)|$,
  we see that the number of nonzeros
  in $\Gamma_i(x)$ is at least $|W^i_{A,\theta}|$.
\end{proof}

We will say that an eigenfunction
of a distance-regular graph
(and the corresponding clique bitrade, if any)
\emph{meets the w.d. bound}
if it has exactly
$\sum_{i=0}^{\diam(\Gamma)} |W^i_{A,\theta}|$ nonzeros.

%-------------------------------------
%-------------------------------------
%-------------------------------------
\subsection{Bitrades that meet the w.d. bound}\label{sect:min}
In this section, we are focused on minimum clique bitrades in distance-regular graphs.
\begin{theorem}\label{th:min}
  Let $\Gamma$ be a distance-regular graph
  of degree $k$.
  Under notation and hypothesis
  of Theorem~\ref{th:bitrade},
  the following assertions are equivalent.

  {\rm (a')} $T$ is an $S$-bitrade
  meeting the w.d. bound.

  {\rm (b')} The function $f^T$
  is an eigenfunction of $\Gamma$
  meeting the w.d. bound
  with eigenvalue $-k/s$.

  {\rm (c')} The subgraph $\Gamma^T$
  is a regular isometric subgraph
  with degree $k/s$.
\end{theorem}
\begin{proof}
  (a')$\Leftrightarrow$(b')
  is straightforward from
  (a)$\Leftrightarrow$(b)
  of Theorem~\ref{th:bitrade}
  and the definition
  of the concept
  ``to meet the w.d. bound.''

  (c')$\Rightarrow$(b').
  Assume (c') holds.
  Consider some $x$ from $T_0$.
  By the isometry property,
  $$\Gamma_i(x) \cap (T_0 \cup T_1) =
  \Gamma^T_i(x) \cap (T_0 \cup T_1) =
  \Gamma^T_i(x) \cap T_{i\bmod 2}$$
  (the last equality holds because $\Gamma^T$
  is bipartite with parts $T_0$ and $T_1$).
  It follows that $f^T$
  is either non-negative
  or non-positive on $\Gamma_i(x)$ and
  $|W^i(f^T)| = |\Gamma_i(x) \cap (T_0 \cup T_1)|$.
  Thus, $|T_0 \cup T_1| = \sum_{i=0}^{\diam(\Gamma)}|W^i(f^T)|$,
  and $f^T$ meets the w.d. bound.

  (a',b')$\Rightarrow$(c').
  Assume (b') holds.
  Then for every $x\in T_0 \cup T_1$ and for every $i$,
  the function $f^T$
  is either non-negative,
  or non-positive on $\Gamma_i(x)$.
  That is,
  $\Gamma_i(x)$ does not intersect
  with either $T_0$ or $T_1$.
  Let us prove by induction on
  $\mathrm{d}_{\Gamma}(x,y)$ that
  $$\mathrm{d}_{\Gamma^T}(x,y)=\mathrm{d}_{\Gamma}(x,y)$$
  for every $x$, $y\in T_0 \cup T_1$.
  For $\mathrm{d}_{\Gamma}(x,y)=0$, this is trivial.
  Let $x$, $y\in T_0$ and $\mathrm{d}_{\Gamma}(x,y)=i$
  (the case when $x$, or $y$,
  or both belong to $T_1$ is similar).
  There is a vertex $v$ in $\Gamma_{i-1}(x)\cap \Gamma_1(y)$.
  A clique from $S$ that contains both $v$ and $x$
  has a vertex $z$ from $T_1$.
  All the vertices
  of this clique
  lie in $\Gamma_{i-1}(x) \cup \Gamma_{i}(x)$.
  But $z$ cannot belong to $\Gamma_{i}(x)$
  as $\Gamma_{i}(x)$ already contains
  a vertex from $T_0$.
  Hence, $z\in \Gamma_{i-1}(x)$.
  By the induction hypothesis,
  $\mathrm{d}_{\Gamma^T}(x,z)=i-1$.
  Therefore, $\mathrm{d}_{\Gamma^T}(x,y)=(i-1)+1=\mathrm{d}_{\Gamma}(x,y)$,
  which proves the statement.
\end{proof}
%--------------------------------
%--------------------------------
%--------------------------------
\begin{theorem}\label{th:d-r}
Assume that, under the notation and the hypothesis
of Theorem~\ref{th:min},
{\rm (a')}--{\rm (c')} hold.
Then the graph $\Gamma^T$ is distance-regular.
Moreover, for every vertex $x$ of $\Gamma^T$ and every $i$
it holds $|\Gamma^T_i(x)|=|W^i_{A,\theta}|$
where $W^i_{A,\theta}$ satisfies (\ref{eq:rekur})
with the intersection numbers of $\Gamma$.
\end{theorem}
\begin{proof}
Consider vertices $x$ and $y$ of $\Gamma^T$ at distance $i$ from each other.
Without loss of generality we assume $y\in T_0$.
We know that the cardinality of $\Gamma_{i+1}(x)\cap \Gamma_1(y)$ is $b_i$,
the corresponding intersection number of $\Gamma$.
Every vertex from this set is in $m$ cliques of $S$ containing $y$.
On the other hand, every such clique contains exactly $s_{i+1}^-$ vertices
of $\Gamma_{i+1}(x)\cap \Gamma_1(y)$. So, the number of cliques containing $y$
and intersecting with $\Gamma_{i+1}(x)$ is $mb_i/s_{i+1}^-$.
Each such clique
contains one element from $T_1$,
and this element lies in $\Gamma_{i+1}(x)$,
because the considered bitrade meets the w.d. bound.
On the other hand, every such element belongs to $m$ cliques containing $y$.
So, $|\Gamma_{i+1}(x)\cap \Gamma_1(y) \cap T_1| = (mb_i/s_{i+1}^-)/m = b_i/s_{i+1}^-$.
By the isometry property, we have $|\Gamma^T_{i+1}(x)\cap \Gamma^T_1(y)| = b_i/s_{i+1}^-$.
Similarly, $|\Gamma^T_{i-1}(x)\cap \Gamma^T_1(y)| = c_i/s_{i-1}^+$,
and the graph $\Gamma^T$ is distance-regular by the definition.

The last statement of the theorem follows from Theorem~\ref{th:min}(b')
and the proof of Corollary~\ref{cor:bound}.
\end{proof}

\begin{corollary}
 For every distance-regular graph $\Gamma$ admitting a Delsarte pair,
there is a sequence
$A'=(b'_0,\ldots,\linebreak[1]b'_{\diam(\Gamma)-1};\linebreak[2]c'_1,\ldots,c'_{\diam(\Gamma)})$
% where $b'_0=\linebreak[1]b'_1+c'_1=\ldots \linebreak[2] = b'_{\diam(\Gamma)-1}+c'_{\diam(\Gamma)-1}\linebreak[1]=c'_{\diam(\Gamma)}$,
such that the existence
of a clique bitrade meeting the w.d. bound in $\Gamma$
is equivalent to the existence
of an isometric distance-regular
subgraph with intersection array $A'$.
\end{corollary}

In the following sections,
we will consider examples of such subgraph.
%===========================================
%===========================================
%===========================================
%===========================================
%===========================================
\section{Known examples}\label{sect:known}
By a \emph{clique design},
we  mean a set of vertices
that has exactly one vertex in common
with each clique of $S$,
given a Delsarte pair $(\Gamma,S)$.
The difference couple %%% couple->pair
$(D_1\backslash D_2,D_2\backslash D_1)$
of two different clique designs
is always a clique bitrade,
while the existence of a clique trade
does not imply
the existence of a clique design
in the same graph.
In this section,
we will consider classes
of distance-regular graphs
for which the theory
of clique designs and clique trades,
in different notations,
is more-or-less developed
in the corresponding
areas of mathematics.

\begin{example}\label{x:octa}
 We start with a very simple example,
when the graph is an $n$-dimensional octahedron,
a regular graph with $2n$ vertices of degree $2n-2$.
There are $2^n$ maximum cliques of cardinality $n$;
a clique design consists of two non-adjacent vertices;
a minimum bitrade corresponds to a square subgraph.
A less trivial problem is to characterize all
$(n,m)$ systems of cliques (for different $m$).
One can find that such systems are
in one-to-one correspondence with
the Boolean functions with $4m$ ones
whose correlation immunity \cite{Siegenthaler:84}
is at least $2$.
\end{example}

\begin{example}\label{x:Hamming}
 The vertex set of the
\emph{Hamming graph} $H(n,q)$
is the set $\{0,\ldots,q-1\}^n$
of words of length $n$
over the alphabet $\{0,\ldots,q-1\}$.
The graph $H(n,2)$
is also known as the \emph{$n$-cube},
or the \emph{hypercube} of dimension $n$.
Two words are adjacent whenever they differ
in exactly one position.
The clique designs in Hamming graphs
are known as the \emph{latin hypercubes}
(in coding theory, these objects are known as the \emph{distance-$2$ MDS codes}),
and the clique bitrades,
as the \emph{latin bitrades}
\cite{Potapov:2013:trade}.
The most studied case,
which corresponds to the latin squares,
is $n=3$, see e.g. \cite{Cav:rev}.
The graph corresponding to a minimum bitrade
is $H(n,2)$ \cite{Potapov:2013:trade}.
\end{example}

\begin{example}\label{x:Johnson}
 The vertices of the
\emph{Johnson graph} $J(n,w)$
are the $w$-subsets of a given set $N$ of cardinality $n$.
Two different vertices are adjacent whenever they intersect
in $w-1$ elements.
The graphs $J(n,w)$ and $J(n,n-w)$ are isomorphic,
and below we assume $2w\le n$.
A \emph{Steiner $\mathrm{S}(w-1,w,n)$ system} $S$ is defined
as a set of vertices of $J(n,w)$, usually called \emph{blocks},
such that every $(w-1)$-subset of $N$ is included in exactly one block from $S$
(see e.g. \cite{ColMat:Steiner}).
It is easy to see that the set of
$w$-subsets of $N$ that include a given $(w-1)$-subset
is a maximum clique in $J(n,w)$.
So, the Steiner $\mathrm{S}(w-1,w,n)$ systems are the clique designs in $J(n,w)$.
The clique bitrades in $J(n,w)$ are known as the
\emph{Steiner $\mathrm{T}(w-1,w,n)$ bitrades},
(in an alternative terminology, Steiner $\mathrm{T}(w-1,w,n)$ trades)
see e.g. \cite{HedKho:trades}.
Any minimum bitrade has the form
\begin{eqnarray*}
\left(\left\{\{a_1^{b_1},\ldots,a_w^{b_w}\} \mid b_1,\ldots,b_w \in\{0,1\}, b_1+\ldots+b_w \equiv 0 \bmod 2 \right\},\right.\\
 \left.\left\{\{a_1^{b_1},\ldots,a_w^{b_w}\} \mid b_1,\ldots,b_w \in\{0,1\}, b_1+\ldots+b_w \equiv 1 \bmod 2 \right\}\right),
\end{eqnarray*}
where $a_1^0,\ldots,a_w^0,a_1^1,\ldots,a_w^1$ are distinct elements of $N$.
The corresponding subgraph is $H(w,2)$.
The minimum bitrade cardinality was found in \cite{Hwang:86}.
In the case $w=3$,
the minimum trade is known
as the Pasch configuration, or the quadrilateral.
\end{example}

%\begin{example}\label{x:Johnson} 
% The vertices of the 
%\emph{Johnson graph} $J(n,w)$ 
%are the binary words of length $n$ 
%and weight (the number of ones) $w$.
%Two words are adjacent whenever they differ 
%in exactly two positions.
%The graphs $J(n,w)$ and $J(n,n-w)$ are isomorphic,
%and below we assume $2w\le n$.
%The clique designs in Johnson graphs 
%are known as the \emph{Steiner $\mathrm{S}(w-1,w,n)$ systems}
%(see e.g. \cite{ColMat:Steiner}),
%and the clique bitrades, 
%as the \emph{Steiner $\mathrm{T}(w-1,w,n)$ bitrades}, 
%alternatively, Steiner $\mathrm{T}(w-1,w,n)$ trades,
%see e.g. \cite{HedKho:trades}.
%The subgraph corresponding to a minimum bitrade
%is $H(w,2)$; 
%an example of the vertex set 
%of such subgraph is 
%$\{(x,\bar x,0,...,0)\mid x,\bar x\in\{0,1\}^w,
%\ \mbox{$\bar x$ is opposite to $x$}\}$.
%The minimum bitrade cardinality was found in \cite{Hwang:86}.
%In the case $w=3$, 
%the minimum trade is known 
%as the Pasch configuration, or the quadrilateral.
%\end{example}

\begin{example}\label{x:halved}
 The vertices of the
\emph{halved $n$-cube}
are the even-weight binary words of length $n$
(i.e., a part of the bipartite $n$-cube).
Two words are adjacent whenever they differ
in exactly two positions.
A maximum clique is the set of
binary $n$-words adjacent in $H(n,2)$
to a fixed odd-weight word;
such clique is Delsarte if and only if $n$ is even.
The clique designs in halved $n$-cubes
are the \emph{extended $1$-perfect codes}.
Such codes exist if and only if $n$ is a power of two,
see e.g. \cite{MWS}.
The minimum cardinality $2^{n/2}$ of
a bitrade was found in \cite{EV:94}
(the authors considered a special type of $1$-perfect trades,
but the argument works for the general case;
the $1$-perfect trades in $H(n-1,2)$
are in one-to-one correspondence with
the extended $1$-perfect trades in the halved $n$-cube).
An example of a minimum clique bitrade is
$\{(x,x)\mid x\in \{0,1\}^{n/2}\}$.
The graph corresponding to a minimum bitrade
is $H(n/2,2)$.
\end{example}

\begin{example}\label{x:folded}
If for every vertex $x$ of a distance-regular graph $\Gamma$,
there is exactly one vertex $y$ at distance $d=\diam(\Gamma)$ from $x$,
then identifying all such pairs $x$, $y$ results in
a distance-regular graph of diameter $\lfloor d/2 \rfloor$,
known as the \emph{folded} $\Gamma$.
It is not difficult to see that
the bipartite isometric subgraph $\Gamma^T$
corresponding to a minimum clique bitrade
will be also folded under this operation.
However, the folded $\Gamma^T$ is bipartite if and only if $d$ is odd;
this reflects the fact that the minimum eigenvalue of $\Gamma$ is an eigenvalue of the folded $\Gamma$
if and only if $d$ is even.
Examples are the folded $J(2d,d)$ and the folded halved $H(2d,2)$,
where the corresponding subgraph is the folded $H(d,2)$, $d$ is even.
\end{example}

The next example shows that analogs of the clique bitrades can be considered
even if the graph has no cliques of required cardinality.
\begin{example}\label{x:doob}
The Shrikhande graph can be defined
on the $16$ quaternary pairs from $\mathbb{Z}_4^2$,
where two pairs are adjacent if and only if their element-wise difference
is one of $(0,1)$, $(0,3)$, $(1,0)$, $(3,0)$, $(1,1)$, $(3,3)$.
The Doob graph $D(m,n)$ is the Cartesian product
of $m>0$ copies of the Shrikhande graph
and $n$ copies of the complete graph on $4$ vertices.
This graph is distance regular with the same intersection array as
the Hamming graph $H(2m+n,4)$.
It follows that it has the same minimum eigenvalue $\theta =-2m-n$,
and the w.d. bound on the number of nonzeros of an eigenfunction is the same too,
i.e., $2^{2m+n}$, for $\theta$.
However, the Doob graph does not admit a Delsarte pair;
moreover, Delsarte cliques,
which have cardinality $4$, does not occur in $D(m,0)$.
So, we cannot apply the definition of a clique design.
Nevertheless, we can apply an alternative definition using Theorem~\ref{th:bitrade}(b):
let us say that a pair of two disjoint independent vertex sets is a pseudo-clique bitrade
if the difference of their characteristic functions is an eigenfunction with minimum eigenvalue.
An example of a minimum bitrade is $\{(0,0),(0,1),(0,2),(0,3)\}^m\{0,1\}^n$;
it is not difficult to find that the subgraph generated by any minimum bitrade is $H(2n+m,2)$.
In a same manner, a pseudo-clique design can be defined as an independent completely
regular set with minimum eigenvalue and covering radius $1$.
Such sets are the maximum independent sets
in the Doob graph \cite{Kro:2015:N-MDS-Doob};
we leave constructing an example as an exercise.
\end{example}
From the last 
 example, we see that defining bitrades in terms of eigenfunctions 
 is a more
general approach than in terms of Delsarte cliques. In a similar manner, bitrades
with other eigenvalues can be defined. For example, the bitrades with eigenvalue $-1$
($1$-perfect bitrades) are studied in the theory of $1$-perfect codes, see e.g. \cite{VorKro:2014en}.

\begin{remark}
 One can weaken the notion of a \emph{clique design}
 by defining it as a set
 that intersects with every clique from $S$
 in a constant $\lambda$ number of elements,
 not necessarily $\lambda=1$.
 In spite of weakening the definition, a clique design is still a completely regular set,
 but it is not an independent set if $\lambda>1$.
 In the partial cases corresponding to the considered examples
 (as well as to the case considered in the next section),
 such designs are also studied in the literature.
\end{remark}

%===========================================
%===========================================
%===========================================
%===========================================
%===========================================
\section{Minimum $q$-ary Steiner bitrades}\label{sect:q-ary}
\def\gauss#1#2{\left[#1 \atop #2\right]_q}
Let $F^n_q$ be an $n$-dimensional
vector space over the Galois field $F_q$ of prime-power order $q$.
The \emph{Grassmann graph} $J_q(n,d)$
is defined as follows.
The vertices are the $d$-dimensional subspaces of $F^n_q$.
Two vertices are adjacent whenever they
intersect in a $(d-1)$-dimensional subspace.
The {Grassmann graph} is a distance-transitive graph of degree
$q\gauss{d}{1}\gauss{n-d}{1}$,
where $\gauss ab=\prod_{i=0}^{b-1}\frac{q^{a-i}-1}{q^{i+1}-1} $
see e.g. \cite[Theorem~9.3.3]{Brouwer}.

All vertices that include
a fixed $(d-1)$-dimensional subspace
form a clique of order
$M=\gauss{n-d+1}1$ in $J_q(n,d)$;
if $n\geq 2d$ then this clique is maximum.
We form an $(M,1)$ system $S$
from all cliques that correspond
to a $(d-1)$-dimensional subspace.
A set of vertices that intersects 
 with every clique from $S$
in exactly one vertex is known as a
\emph{$q$-ary Steiner $\mathrm{S}_q(d-1,d,n)$ system}.
Constructing $q$-ary Steiner $\mathrm{S}_q(d-1,d,n)$
systems with $d\ge 3$ is not easy;
at the moment,
only the existence of $\mathrm{S}_2(2,3,13)$
is known in this field \cite{BEOVW:q-Steiner}.
An $S$-bitrade in $J_q(n,d)$ is called
a \emph{Steiner $\mathrm{T}_q(d-1,d,n)$ bitrade}.

Before formulating the main theorem of this section, we briefly introduce the dual polar graph $D_d(q)$ (see, e.g., \cite{Brouwer}), which plays the role of the bipartite subgraph $\Gamma^T$ for $\Gamma=J_q(n,d)$.
Note that the class of dual polar graphs contains several other subclasses \cite[\S 9.4]{Brouwer}, which are not considered here,
but the graphs of type $D_d(q)$ are the only dual polar graphs that are bipartite.

A quadratic form
$Q:F^n_q\rightarrow F_q$
is said to be  \emph{nondegenerate}
if its kernel
$\{x \mid Q(y+x)=Q(y)  \forall y\in F^n_q\}$ is zero.
A subspace $V$ of
$F^n_q$ is called \emph{totally isotropic}
whenever the form vanishes
completely on $V$, i.e., $Q(V)=\{0\}$.
The maximum dimension of a totally
isotropic subspace is known as
the \emph{Witt index} of $Q$.
If $n=2d$, then the maximum
Witt index of a nondegenerate quadratic form is equal to $d$.
There exists a unique
(up to isomorphism)
nondegenerate quadratic form with the Witt index $d$.
One of its representations is
$Q_0(v_1,\dots,v_d,u_1,\dots u_d)=v_1u_1+\dots+v_du_d$.
The \emph{dual polar graph} $D_d(q)$
has as vertices the $d$-dimensional totally isotropic subspaces,
with respect to $Q_0$;
two vertices $\alpha$ and $\beta$
are adjacent whenever
$\dim (\alpha\cap\beta)= d-1$.

\begin{theorem}\label{th:q-trade}
 The minimum cardinality
 of a \emph{Steiner $\mathrm{T}_q(d-1,d,n\ge 2d)$ bitrade}
 is
\begin{equation}\label{eq:q-trade}
 \prod\limits_{i=1}^{d}(q^{d-i}+1)=\sum\limits_{i=0}^{d}q^{i\choose
2}\gauss{d}{i},
\end{equation}
 which is also equal to the value of the w.d. bound.

 The bipartite distance-regular subgraph of $J_q(n,d)$
 generated by a Steiner $\mathrm{T}_q(d-1,d,n)$ bitrade 
  has the parameters of the dual polar graph $D_d(q)$.
\end{theorem}

\begin{proof}
$J_q(2d,d)$ is an isometric subgraph of $J_q(n,d)$;
$D_d(q)$ is an isometric subgraph
of $J_q(2d,d)$ \cite[p.276]{Brouwer}.
$D_d(q)$ is
a bipartite distance-regular graph
of degree $(q^d -1)/(q-1)$
(the biparticity and the degree are easily retrieved
from the intersection array \cite[Theorem 9.4.3]{Brouwer})
and order $\prod\limits_{i=1}^{d}(q^{d-i}+1)$
\cite[p.274, Lemma 9.4.1]{Brouwer}.
A proof of the identity (\ref{eq:q-trade}) can be found in \cite[Equation (1.87)]{Stanley}.
Since $(q^d -1)/(q-1) = k(M-1)$ with $k=q\gauss{d}{1}\gauss{n-d}{1}$ and $M=\gauss{n-d+1}1$,
the result follows from Theorem~\ref{th:min}.
\end{proof}
\begin{remark}
  It can be found that the $i$th summand
  $S_i=q^{i\choose 2}\gauss{d}{i}$ of the right part of (\ref{eq:q-trade})
  coincides with the number $|D_d(q)_i(x)|$
  of vertices at distance $i$ from a fixed vertex $x$ in $D_d(q)$.
  A
   straightforward way to prove
    this is checking the relation
  $b'_{i-1}S_{i-1} = c'_iS_i$ where $b'_{i}=q^i\gauss{d-i}1$ and $c'_{i}=\gauss{i}1$ are coefficients from the
  intersection array of $D_d(q)$, which can be found in \cite[Theorem 9.4.3]{Brouwer}
  (this relation correspond
   to double-counting
   the edges between
   $D_d(q)_{i-1}(x)$
   and $D_d(q)_i(x)$).
%   Another way is to check the relations (\ref{eq:rekur})
%   with $W^{i}_{A,\theta} = (-1)^iS_i$,
%   where, $b_{i}=q^{2i+1}\gauss{d-i}1\gauss{n-d-i}1$ and $c_i=\gauss{i}1^2$ 
%   are from the intersection array of the Grassmann graph $J_q(n,d)$
%   \cite[Theorem 9.3.3]{Brouwer}, \ 
%   $\theta=-\frac{b_0}{M-1}=-q\gauss{d}{1}\gauss{n-d}{1} / (\gauss{n-d+1}1-1)$.
\end{remark}
\begin{remark}
The minimum $\mathrm{T}_q(2,3,n)$ trades
can be considered
as $q$-ary analogs
of the Pasch configuration (quadrilateral).
In particular, the formula $(q+1)(q^2+1) = 15$, $40$, $85$, $156$, ... ($q=2,3,4,5,...$)
for the size of a minimum trade
(which is the half of the size
of a minimum bitrade (\ref{eq:q-trade}))
is satisfied by the Pasch configuration with $q=1$.
As in the case of the Pasch configuration, the graph $J_q(n,3)$, if $n$ is large enough, contains many isomorphic copies
of the minimum $\mathrm{T}_q(2,3,n)$  trade that lie at distance
more than $1$ from each other (so, simple metrical arguments do not forbid them to belong the same $q$-ary Steiner system).
Indeed, each $6$-dimensional subspace of $F_q^n$ corresponds to a subgraph isomorphic to $J_q(6,3)$, which has a subgraph isomorphic to $D_3(q)$.
If two such subspaces have no common $2$-dimensional subspace, then the corresponding subgraphs are mutually independent.
However, the question how many ($0$, $1$, very few, or good many) minimum trades a real $q$-ary Steiner system can include remains untouched.

Another representation
of the $\mathrm{T}_q(2,3,n)$ trades
constructed in the current paper
was announced in \cite{Mog:Maltsev14}.
\end{remark}

\begin{problem}
  The following question is natural: is a minimum Steiner $\mathrm{T}_q(d-1,d,n)$ bitrade unique,
  up to isomorphism of the Grassmann graph?
  As noted in \cite[Remark 9.4.6]{Brouwer},
  in general, the dual polar graph $D_d(q)$ is not unique
  as a distance regular graph with given intersection array.
  The question is if there are nonisomorphic isometric embeddings of such graphs
  into the Grassmann graph.
  Note that the minimum trades from the examples of Section~\ref{sect:known}
  are known to be unique.
\end{problem}

\section*{Acknowledgements}
This research was funded
 by the Russian Science Foundation (grant No 14-11-00555).
%\bibliographystyle{plain}
%\bibliography{../../k}
\providecommand\href[2]{#2} \providecommand\url[1]{\href{#1}{#1}}
  \providecommand\bblmay{May} \providecommand\bbloct{October}
  \providecommand\bblsep{September} \def\DOI#1{{\small {DOI}:
  \href{http://dx.doi.org/#1}{#1}}}\def\DOIURL#1#2{{\small{DOI}:
  \href{http://dx.doi.org/#2}{#1}}}\providecommand\bbljun{June}

\end{document}